\documentclass[10pt]{article}

\usepackage{amsmath,amsthm,amsfonts,latexsym,amsopn,verbatim,amscd,amssymb}

\theoremstyle{plain}
\newtheorem{theorem}{Theorem}[section]

\newtheorem{proposition}[theorem]{Proposition}
\newtheorem{corollary}[theorem]{Corollary}
\theoremstyle{definition}
\numberwithin{equation}{section}

\DeclareMathOperator{\L-spec}{L-Spec}

\DeclareMathOperator{\cent}{Cent}

\newcommand{\bnum}{\begin{enumerate}}
\newcommand{\enum}{\end{enumerate}}

\begin{document}

\title{On Laplacian energy of non-commuting graphs of  finite groups }
\author{Parama Dutta and Rajat Kanti Nath\footnote{Corresponding author}}
\date{}
\maketitle
\begin{center}\small{ Department of Mathematical Sciences,\\ Tezpur
University,  Napaam-784028, Sonitpur, Assam, India.\\
Emails:   parama@gonitsora.com and  rajatkantinath@yahoo.com}
\end{center}

%\thanks{}

\smallskip

\noindent {\small{\textbf{Abstract:}  In this paper, we compute Laplacian energy of the non-commuting graphs of some classes of finite non-abelian groups.
}}

\bigskip



\section{Introduction} \label{S:intro}
Let ${\mathcal{G}}$ be a graph. Let $A({\mathcal{G}})$ and $D({\mathcal{G}})$ denote the adjacency matrix  and degree matrix of the graph respectively. Then the Laplacian matrix  of ${\mathcal{G}}$ is given by  $L({\mathcal{G}})  = D({\mathcal{G}}) - A({\mathcal{G}})$. Let  $\beta_1,  \beta_2, \dots, \beta_m$ be the eigenvalues of  $L({\mathcal{G}})$ with multiplicities $b_1, b_2, \dots, b_m$. Then the Laplacian spectrum of  ${\mathcal{G}}$, denoted by $\L-spec({\mathcal{G}})$, is   the set $\{\beta_1^{b_1}, \beta_2^{b_2}, \dots, \beta_m^{b_m}\}$. The Laplacian energy of ${\mathcal{G}}$,   denoted by $LE({\mathcal{G}})$,   is given by

\begin{equation}\label{Lenergy}
LE({\mathcal{G}}) = \sum_{\mu \in \L-spec({\mathcal{G}})}\left|\mu - \frac{2|e({\mathcal{G}})|}{|v({\mathcal{G}})|}\right| 
\end{equation}
where $v({\mathcal{G}})$ and $e({\mathcal{G}})$ are the sets of  vertices    and edges   of the graph $\mathcal{G}$ respectively.  A graph $\mathcal{G}$ is called L-integral if  $\L-spec({\mathcal{G}})$ contains only integers. Various properties of L-integral graphs and $LE({\mathcal{G}})$ are studied in \cite{Abreu08,Kirkland07,Merries94}.

Let $G$ be a finite non-abelian group with center $Z(G)$. The non-commuting graph of  $G$, denoted by ${\mathcal{A}}_G$  is a simple undirected graph such that $v({\mathcal{A}}_G) = G\setminus Z(G)$ and two vertices $x$ and $y$ are adjacent if and only if $xy \ne yx$. Various aspects of non-commuting graphs of different families of finite non-abelian groups are studied  in  \cite{Ab06,AF14,dbb10,Abd13,tal08}. Note that the complement of   ${\mathcal{A}}_G$ is the commuting graph of  $G$ denoted by ${\overline{\mathcal{A}}}_G$. Commuting graphs are studied extensively in \cite{amr06,iJ07,mP13,par13}. In \cite{lspac}, the authors have computed the Laplacian  spectrum of the non-commuting graphs of  several  well-known families finite non-abelian groups. In this paper we study the Laplacian energy of those classes of finite groups.

\section{Some Computations}
In this section, we compute Laplacian energy of some families of groups whose central factors are some well-known  groups.
\begin{theorem} \label{order-20}
Let $G$ be a finite group and $\frac{G}{Z(G)} \cong Sz(2)$, where $Sz(2)$ is the Suzuki group presented by $\langle a, b : a^5 = b^4 = 1, b^{-1}ab = a^2 \rangle$. Then

\[
LE({\mathcal{A}}_G) = \left(\frac{120}{19}|Z(G)| + 30\right)|Z(G)|.
\]
\end{theorem}
\begin{proof}
It is clear that $|v({\mathcal{A}}_G)| = 19|Z(G)|$. Since $\frac{G}{Z(G)} \cong Sz(2)$, we have
\[
\frac{G}{Z(G)} = \langle aZ(G), bZ(G) : a^5Z(G) = b^4Z(G) = Z(G), b^{-1}abZ(G) = a^2Z(G) \rangle.
\]
Then
\[
\begin{array}{ll}
C_G(a)  &= Z(G)\sqcup aZ(G) \sqcup a^2Z(G)\sqcup a^3Z(G)\sqcup a^4Z(G),\\
C_G(ab) &= Z(G)\sqcup abZ(G) \sqcup a^4b^2Z(G)\sqcup a^3b^3Z(G),\\
C_G(a^2b) &= Z(G)\sqcup a^2bZ(G) \sqcup a^3b^2Z(G)\sqcup ab^3Z(G),\\
C_G(a^2b^3) &= Z(G)\sqcup a^2b^3Z(G) \sqcup ab^2Z(G)\sqcup a^4bZ(G),\\
C_G(b) &= Z(G)\sqcup bZ(G) \sqcup b^2Z(G)\sqcup b^3Z(G) \quad \text{ and }\\
C_G(a^3b) &= Z(G)\sqcup a^3bZ(G) \sqcup a^2b^2Z(G)\sqcup a^4b^3Z(G)
\end{array}
\]
are the only centralizers of non-central elements of $G$. Since all these distinct centralizers are abelian, we have
\[
{\overline{{\mathcal{A}}}_G} = K_{4|Z(G)|}\sqcup 5K_{3|Z(G)|}
\]
and hence $|e({\mathcal{A}}_G)| = 150{|Z(G)|}^2$. By Theorem 3.1 of \cite{lspac}, we have
\[
\L-spec({\mathcal{A}}_G) =\{  0,  {(15|Z(G)|)}^{4|Z(G)|-1},  {(16|Z(G)|)}^{15|Z(G)|-5},  {(19|Z(G)|)}^5 \}.
\]
Therefore, $\left|0 - \frac{2|e({\mathcal{A}}_G)|}{|v({\mathcal{A}}_G)|}\right| = \frac{300}{19}|Z(G)|$, $\left|15|Z(G)| - \frac{2|e({\mathcal{A}}_G)|}{|v({\mathcal{A}}_G)|}\right| = \frac{15}{19}|Z(G)|$,\\ $\left|16|Z(G)| - \frac{2|e({\mathcal{A}}_G)|}{|v({\mathcal{A}}_G)|}\right| = \frac{4}{19}|Z(G)|$, $\left|19|Z(G)| - \frac{2|e({\mathcal{A}}_G)|}{|v({\mathcal{A}}_G)|}\right| = \frac{61}{19}|Z(G)|$. By \eqref{Lenergy}, we have
\begin{align*}
  LE({\mathcal{A}}_G) = & \frac{300}{19}|Z(G)| + (4|Z(G)|-1)\left(\frac{15}{19}|Z(G)|\right) + (15|Z(G)|-5)\left(\frac{4}{19}|Z(G)|\right)\\
   & + 5\left(\frac{61}{19}|Z(G)|\right).
\end{align*}
Hence the result follows.

\end{proof}

\begin{theorem}\label{main2}
Let $G$ be a finite group such that $\frac{G}{Z(G)} \cong {\mathbb{Z}}_p \times {\mathbb{Z}}_p$, where $p$ is a prime integer. Then

\[
LE({\mathcal{A}}_G) = 2p(p - 1)|Z(G)|.
\]

\end{theorem}
\begin{proof}
It is clear that $|v({\mathcal{A}}_G)| = (p^2 - 1)|Z(G)|$. Since $\frac{G}{Z(G)} \cong \mathbb{Z}_p\times \mathbb{Z}_p$, we have
$
\frac{G}{Z(G)} = \langle aZ(G), bZ(G) : a^p, b^p, aba^{-1}b^{-1} \in Z(G)\rangle,
$
where $a$, $b\in$ $G$ with $ab\neq ba$. Then for any $z\in Z(G)$
\begin{align*}
  C_G(a) & = Z(G)\sqcup aZ(G)\sqcup \dots \sqcup a^{p - 1}Z(G) \text{ for } 1\leq i\leq p - 1 \quad \text{and} \\
  C_G(a^jb) &= Z(G)\sqcup a^jbZ(G)\sqcup \dots \sqcup {a^jb}^{p - 1}Z(G) \text{ for } 1\leq j\leq p
\end{align*}
are the only centralizers of non-central elements of $G$. Also note that these
centralizers are abelian subgroups of $G$. Therefore
\[
{\overline{{\mathcal{A}}}_G} = K_{|C_G(a)\setminus Z(G)|}\sqcup (\underset{j = 1}{\overset{p}{\sqcup}} K_{|C_G(a) \setminus Z(G)|}).
\]
Since, $|C_G(a)|$ and  $|C_G(a^jb)|$ for $1\leq j\leq p$, we have

\[
{\overline{{\mathcal{A}}}_G} = (p + 1)K_{(p - 1)|Z(G)|}.
\]
and hence $|e({\mathcal{A}}_G)| = \frac{p(p + 1){(p - 1)}^2}{2}{|Z(G)|}^2$. By Theorem 3.2 of \cite{lspac}, we have
\[
\L-spec({\mathcal{A}}_G) = \{ 0,  {((p^2-p)|Z(G)|)}^{(p^2-1)|Z(G)|-p-1},  {((p^2-1)|Z(G)|)}^p \}.
\]
Therefore, $\left|0 - \frac{2|e({\mathcal{A}}_G)|}{|v({\mathcal{A}}_G)|}\right| = p(p - 1)|Z(G)|$, $\left|(p^2-p)|Z(G)| - \frac{2|e({\mathcal{A}}_G)|}{|v({\mathcal{A}}_G)|}\right| = 0$ and \\ $\left|(p^2-1)|Z(G)| - \frac{2|e({\mathcal{A}}_G)|}{|v({\mathcal{A}}_G)|}\right| = (p - 1)|Z(G)|$.  By \eqref{Lenergy}, we have
\[
LE({\mathcal{A}}_G) = p(p - 1)|Z(G)| + ((p^2-1)|Z(G)|-p-1)0 + p((p - 1)|Z(G)|).
\]
Hence the result follows.

\end{proof}

\begin{corollary}
Let $G$ be a non-abelian group of order $p^3$, for any prime $p$, then
\[
LE({\mathcal{A}}_G) = 2p^2(p - 1).
\]
\end{corollary}

\begin{proof}
Note that $|Z(G)| = p$ and  $\frac{G}{Z(G)} \cong {\mathbb{Z}}_p \times {\mathbb{Z}}_p$. Hence the  result follows from Theorem \ref{main2}.
\end{proof}

\begin{theorem}\label{main4}
Let $G$ be a finite group such that $\frac{G}{Z(G)} \cong D_{2m}$, for $m \geq 2$. Then
\[
LE({\mathcal{A}}_G) = \frac{(2m^2 - 3m)(m - 1){|Z(G)|}^2 + m(4m - 3)|Z(G)|}{2m - 1}.
\]
\end{theorem}
\begin{proof}
Clearly, $|v({\mathcal{A}}_G)| = (2m - 1)|Z(G)|$. Since $\frac{G}{Z(G)} \cong D_{2m}$ we have $\frac{G}{Z(G)} = \langle xZ(G), yZ(G) : x^2, y^m,  xyx^{-1}y\in Z(G)\rangle$, where $x, y \in G$ with $xy \ne yx$.
It is easy to see that  for any $z \in Z(G)$
\begin{align*}
  C_G(xy^j) & = C_G(xy^jz) = Z(G)  \sqcup xy^jZ(G), 1 \leq j \leq m \quad \text{and}  \\
  C_G(y) & = C_G(y^iz) = Z(G) \sqcup yZ(G) \sqcup\cdots \sqcup  y^{m - 1}Z(G), 1 \leq i \leq m - 1
\end{align*}
are the only  centralizers of non-central elements of $G$. Also note that these centralizers are abelian subgroups of $G$ and  $|C_G(x^jy)| = 2|Z(G)|$ for $1 \leq j \leq m$ and $|C_G(y)| = m|Z(G)|$. Hence
\[
{\overline{{\mathcal{A}}}_G} = K_{(m - 1)|Z(G)|}\sqcup mK_{|Z(G)|}.
\]
and $|e({\mathcal{A}}_G)| = \frac{3m(m - 1){|Z(G)|}^2}{2}$. By Theorem 3.4 of \cite{lspac}, we have
\begin{align*}
\L-spec({\mathcal{A}_G}) = &\{ 0 , {(m|Z(G)|)}^{(m-1)|Z(G)|-1} , {(2(m-1)|Z(G)|)}^{m|Z(G)|-m},\\
&{((2m-1)|Z(G)|)}^m \}.
\end{align*}
Therefore, $\left|0 - \frac{2|e({\mathcal{A}}_G)|}{|v({\mathcal{A}}_G)|}\right| = \frac{3m(m - 1)|Z(G)|}{2m - 1}$, $\left|m|Z(G)| - \frac{2|e({\mathcal{A}}_G)|}{|v({\mathcal{A}}_G)|}\right| = \frac{m(m - 1)|Z(G)|}{2m - 1}$,\\ $\left|2(m-1)|Z(G)| - \frac{2|e({\mathcal{A}}_G)|}{|v({\mathcal{A}}_G)|}\right| = \frac{(m - 1)(m - 2)|Z(G)|}{2m - 1}$ and $\left|(2m-1)|Z(G)| - \frac{2|e({\mathcal{A}}_G)|}{|v({\mathcal{A}}_G)|}\right| = \frac{(m^2 - m +1)|Z(G)|}{2m - 1}$. By \eqref{Lenergy}, we have
\begin{align*}
  LE({\mathcal{A}}_G) = & \frac{3m(m - 1)|Z(G)|}{2m - 1} + ((m-1)|Z(G)|-1)\left(\frac{m(m - 1)|Z(G)|}{2m - 1}\right)\\
   & + (m|Z(G)|-m)\left(\frac{(m - 1)(m - 2)|Z(G)|}{2m - 1}\right) + m\left(\frac{(m^2 - m +1)|Z(G)|}{2m - 1}\right)
\end{align*}
and hence the result follows.
\end{proof}

Using Theorem \ref{main4}, we now compute the  Laplacian energy  of the non-commuting graphs of the groups $M_{2mn}, D_{2m}$ and $Q_{4n}$ respectively.

\begin{corollary}\label{main05}
Let $M_{2mn} = \langle a, b : a^m = b^{2n} = 1, bab^{-1} = a^{-1} \rangle$ be a metacyclic group, where $m > 2$.
\[
LE({\mathcal{A}}_{M_{2mn}}) =
\begin{cases}
  \frac{m(2m - 3)(m - 1)n^2 + m(4m - 3)n}{2m - 1}, & \mbox{if m is odd }  \\
  \frac{m(m - 2)(m - 3)n^2 + m(2m - 3)n}{m - 1}, & \mbox{if m is even}.
\end{cases}
\]
\end{corollary}`

\begin{proof}
  Observe that $Z(M_{2mn}) = \langle b^2 \rangle$ or $\langle b^2 \rangle \cup a^{\frac{m}{2}}\langle b^2 \rangle$ according as $m$ is odd or even.  Also, it is easy to see that $\frac{M_{2mn}}{Z(M_{2mn})} \cong D_{2m}$ or $D_m$ according as $m$ is odd or even. Hence, the result follows from Theorem \ref{main4}
\end{proof}
\noindent As a corollary to the above result we have the following results.

\begin{corollary}\label{main005}
Let $D_{2m} = \langle a, b : a^m = b^{2} = 1, bab^{-1} = a^{-1} \rangle$ be  the dihedral group of order $2m$, where $m > 2$.
\[
LE({\mathcal{A}}_{D_{2m}}) =
\begin{cases}
  m^2, & \mbox{ if m is odd } \\
  \frac{m(m^2 - 3m + 3)}{m - 1}, & \mbox{if m is even }.
\end{cases}
\]

\end{corollary}

\begin{corollary}\label{q4m}
Let $Q_{4m} = \langle x, y : y^{2m} = 1, x^2 = y^m,yxy^{-1} = y^{-1}\rangle$, where $m \geq 2$, be the   generalized quaternion group of order $4m$. Then
\[
LE({\mathcal{A}}_{Q_{4m}}) = \frac{2m(4m^2 - 6m + 3)}{2m - 1}.
\]
\end{corollary}

\begin{proof}
The result follows from Theorem \ref{main4} noting that  $Z(Q_{4m}) = \{1, a^m\}$ and  $\frac{Q_{4m}}{Z(Q_{4m})} \cong D_{2m}$.
\end{proof}

\section{Some well-known groups}
Now  we compute Laplacian energy of  the  non-commuting graphs of some  well-known families of finite non-abelian groups.
\begin{proposition}\label{order-pq}
Let $G$ be a non-abelian group of order $pq$, where $p$ and $q$ are primes with $p\mid (q - 1)$. Then
\[
LE({\mathcal{A}}_G) = \frac{2q(p^2 - 1)(q - 1)}{pq - 1}.
\]

\end{proposition}
\begin{proof}
It is clear that $|v({\mathcal{A}}_G)| = pq - 1$. Note that $|Z(G)| = 1$ and the centralizers of non-central elements of $G$ are precisely the Sylow subgroups of $G$. The number of Sylow $q$-subgroups and Sylow $p$-subgroups of $G$ are one and $q$ respectively. Therefore, we have
\[
{\overline{{\mathcal{A}}}_G} = K_{q-1} \sqcup qK_{p - 1}
\]
and hence $|e({\mathcal{A}}_G)| = \frac{q(p^2 - 1)(q - 1)}{2}$. By Proposition 4.1 of \cite{lspac}, we have
\[
\L-spec({\mathcal{A}}_G) = \{0, {(pq - q)}^{q - 2}, {(pq - p)}^{pq - 2q}, {(pq - 1)^q}\}.
\]
Therefore, $\left|0 - \frac{2|e({\mathcal{A}}_G)|}{|v({\mathcal{A}}_G)|}\right| = \frac{p^2q^2 - p^2q - q^2 + q}{pq - 1}$, $\left|pq - q - \frac{2|e({\mathcal{A}}_G)|}{|v({\mathcal{A}}_G)|}\right| = \frac{q(q - p)(p - 1)}{pq - 1}$,\\ $\left|pq - p - \frac{2|e({\mathcal{A}}_G)|}{|v({\mathcal{A}}_G)|}\right| = \frac{(q - p)(q - 1)}{pq - 1}$ and $\left|pq - 1 - \frac{2|e({\mathcal{A}}_G)|}{|v({\mathcal{A}}_G)|}\right| = \frac{p^2q + q^2 - 2pq -q + 1}{pq - 1}$. By \eqref{Lenergy}, we have
\begin{align*}
 LE({\mathcal{A}}_G)  = & \frac{p^2q^2 - p^2q - q^2 + q}{pq - 1} + (q - 2)\left(\frac{q(q - p)(p - 1)}{pq - 1}\right) \\
  & + (pq - 2q)\left(\frac{(q - p)(q - 1)}{pq - 1}\right) + q\left(\frac{p^2q + q^2 - 2pq -q + 1}{pq - 1}\right)
\end{align*}
and hence the result follows.
\end{proof}
\begin{proposition}\label{semid}
Let $QD_{2^n}$ denotes the quasidihedral group $\langle a, b : a^{2^{n-1}} =  b^2 = 1, bab^{-1} = a^{2^{n - 2} - 1}\rangle$, where $n \geq 4$. Then
\[
 LE({\mathcal{A}}_{QD_{2^n}}) = \frac{2^{3n - 3} - 2^{2n} + 3.2^n}{2^{n -1} - 1}.
\]
\end{proposition}
\begin{proof}
It is clear that $Z(QD_{2^n}) = \{1, a^{2^{n - 2}}\}$, so $|v({\mathcal{A}}_{QD_{2^n}})| = 2(2^{n - 1} - 1)$. Note that
\begin{align*}
  C_{QD_{2^n}}(a) &= C_{QD_{2^n}}(a^i)  = \langle a \rangle \text{ for } 1 \leq i \leq 2^{n - 1} - 1, i \ne 2^{n - 2} \quad \text{and}\\
  C_{QD_{2^n}}(a^jb) &= \{1, a^{2^{n - 2}}, a^jb, a^{j + 2^{n - 2}}b \} \text{ for } 1 \leq j \leq 2^{n - 2}
\end{align*}
are the only  centralizers of non-central elements of $QD_{2^n}$. Note that these centralizers are abelian subgroups of  $QD_{2^n}$. Therefore, we have
\[
\overline{{\mathcal{A}}}_{QD_{2^n}} = K_{|C_{QD_{2^n}}(a)\setminus Z(QD_{2^n})|} \sqcup (\underset{j = 1}{\overset{2^{n - 2}}{\sqcup}} K_{|C_{QD_{2^n}}(a^jb)\setminus Z(QD_{2^n})|}).
\]
Since $|C_{QD_{2^n}}(a)| = 2^{n - 1}, |C_{QD_{2^n}}(a^jb)| = 4$ for $1 \leq j \leq 2^{n - 2}$, we have $\overline{{\mathcal{A}}}_{QD_{2^n}} = K_{2^{n - 1} - 2} \sqcup 2^{n - 2} K_2$. Hence
\[
|e({\mathcal{A}}_{QD_{2^n}})| = \frac{3.2^{2n - 2} - 6.2^{n - 1}}{2}.
\]
By Proposition 4.2 of \cite{lspac}, we have
\[
\L-spec({\mathcal{A}}_{QD_{2^n}}) = \{0, {(2^{n-1})}^{2^{n-1}-3}, {(2^n-4)}^{2^{n-2}}, {(2^n-2)}^{2^{n-2}}\}.
\]
Therefore,  $\left|0 - \frac{2|e({\mathcal{A}}_{QD_{2^n}})|}{|v({\mathcal{A}}_{QD_{2^n}})|}\right| = \frac{3.2^{n - 1}(2^{n - 1} - 2)}{2.2^{n - 1} - 2}$, $\left|2^{n-1} - \frac{2|e({\mathcal{A}}_{QD_{2^n}})|}{|v({\mathcal{A}}_{QD_{2^n}})|}\right| = \frac{2^{2n - 2} - 4.2^{n - 1}}{2.2^{n - 1} - 2}$,\\ $\left|2^n-4 - \frac{2|e({\mathcal{A}}_{QD_{2^n}})|}{|v({\mathcal{A}}_{QD_{2^n}})|}\right| = \frac{2^{2n - 2} - 6.2^{n - 1} + 8}{2.2^{n - 1} - 2}$ and  $\left|2^n-2 - \frac{2|e({\mathcal{A}}_{QD_{2^n}})|}{|v({\mathcal{A}}_{QD_{2^n}})|}\right| = \frac{2^{2n - 2} - 2.2^{n - 1} + 4}{2.2^{n - 1} - 2}$.
By \eqref{Lenergy}, we have
\begin{align*}
  LE({\mathcal{A}}_{QD_{2^n}}) = & \frac{3.2^{n - 1}(2^{n - 1} - 2)}{2.2^{n - 1} - 2} + (2^{n-1}-3)\left(\frac{2^{2n - 2} - 4.2^{n - 1}}{2.2^{n - 1} - 2}\right) \\
   & + 2^{n-2}\left(\frac{2^{2n - 2} - 6.2^{n - 1} + 8}{2.2^{n - 1} - 2}\right) + 2^{n-2}\left(\frac{2^{2n - 2} - 2.2^{n - 1} + 4}{2.2^{n - 1} - 2}\right)
\end{align*}
and hence the result follows.
\end{proof}

\begin{proposition}\label{psl}
Let $G$ denotes the projective special linear group  $PSL(2, 2^k)$, where $k \geq 2$. Then
\[
 LE({\mathcal{A}}_G) = \frac{3.2^{6k} - 2.2^{5k} - 7.2^{4k} + 2^{3k} + 4.2^{2k} +2^k}{2^{3k} - 2^k - 1}.
\]
\end{proposition}
\begin{proof}
Clearly, $|v({\mathcal{A}}_G)| = 2^{3k} - 2^k - 1$. Since $G$ is a non-abelian group of order $2^k(2^{2k} - 1)$ and its center is trivial. By Proposition 3.21 of \cite{Ab06}, the set of centralizers of non-trivial elements of $G$ is given by
\[
\{xPx^{-1}, xAx^{-1}, xBx^{-1} : x \in G\}
\]
where $P$ is an elementary abelian \quad $2$-subgroup and  $A, \quad B$ are  cyclic subgroups of  $G$ having order $2^k, 2^k - 1$ and $2^k + 1$ respectively. Also the number of conjugates of $P, A$ and $B$ in $G$ are  $2^k + 1, 2^{k - 1}(2^k + 1)$ and $2^{k - 1}(2^k - 1)$ respectively. Hence $\overline{{\mathcal{A}}_{G}}$ of $G$ is given by
\[
%\Gamma_{PSL(2, 2^k)} =
 (2^k + 1)K_{|xPx^{-1}| - 1} \sqcup 2^{k - 1}(2^k + 1)K_{|xAx^{-1}| - 1} \sqcup 2^{k - 1}(2^k - 1)K_{|xBx^{-1}| - 1}.
\]
That is, $\overline{{\mathcal{A}}}_{G} = (2^k + 1)K_{2^k - 1} \sqcup 2^{k - 1}(2^k + 1)K_{2^k - 2} \sqcup 2^{k - 1}(2^k - 1)K_{2^k}$. Therefore,
\[
e({\mathcal{A}}_G) = \frac{2^{6k} - 3.2^{4k} - 2^{3k} + 2.2^{2k} + 2^k}{2}.
\]
By Proposition 4.3 of \cite{lspac}, we have
\begin{align*}
\L-spec({\mathcal{A}}_{G}) = & \{0, {(2^{3k}-2^{k+1}-1)}^{2^{3k-1}-2^{2k}+2^{k-1}}, {(2^{3k}-2^{k+1})}^{2^{2k}-2^k-2},\\
 & {(2^{3k}-2^{k+1}+1)}^{2^{3k-1}-2^{2k}-3.2^{k-1}}, {(2^{3k}-2^k-1)}^{2^{2k}+2^k}\}.
\end{align*}
Now, $\left|0 - \frac{2|e({\mathcal{A}}_{G})|}{|v({\mathcal{A}}_G)|}\right| = \frac{2^{6k} - 3.2^{4k} - 2^{3k} + 2.2^{2k} + 2^k}{2^{3k} - 2^k - 1}$, $\left|2^{3k}-2^{k+1}-1 - \frac{2|e({\mathcal{A}}_G)|}{|v({\mathcal{A}}_{G})|}\right| = \frac{2^{3k} -2.2^k - 1}{2^{3k} - 2^k - 1}$,\\ $\left|2^{3k}-2^{k+1} - \frac{2|e({\mathcal{A}}_{G})|}{|v({\mathcal{A}}_{G})|}\right| = \frac{2^k}{2^{3k} - 2^k - 1}$, $\left|2^{3k}-2^{k+1}+1 - \frac{2|e({\mathcal{A}}_{G})|}{|v({\mathcal{A}}_{G})|}\right| = \frac{2^{3k} - 1}{2^{3k} - 2^k - 1}$\\ and  $\left|2^{3k}-2^k-1 - \frac{2|e({\mathcal{A}}_{G})|}{|v({\mathcal{A}}_{G})|}\right| = \frac{2 ^{4k} - 2^{3k} - 2^{2k} + 2^k + 1}{2^{3k} - 2^k - 1}$. By \eqref{Lenergy}, we have
\begin{align*}
LE({\mathcal{A}}_G) = & \frac{2^{6k} - 3.2^{4k} - 2^{3k} + 2.2^{2k} + 2^k}{2^{3k} - 2^k - 1} + (2^{3k-1}-2^{2k}+2^{k-1})\left(\frac{2^{3k} -2.2^k - 1}{2^{3k} - 2^k - 1}\right) \\
   & + (2^{2k}-2^k-2)\left(\frac{2^k}{2^{3k} - 2^k - 1}\right) + (2^{3k-1}-2^{2k}-3.2^{k-1})\left(\frac{2^{3k} - 1}{2^{3k} - 2^k - 1}\right) \\
   & + (2^{2k}+2^k)\left(\frac{2 ^{4k} - 2^{3k} - 2^{2k} + 2^k + 1}{2^{3k} - 2^k - 1}\right)
\end{align*}
and hence the result follows.
\end{proof}

\begin{proposition}
Let $G$ denotes the  general linear group  $GL(2, q)$, where $q = p^n > 2$ and $p$ is a prime. Then
\[
LE({\mathcal{A}}_G) = \frac{q^9 - 2q^8 - q^7 + 2q^6 + 2q^5 + q^4 - 4q^3 +2q^2 + q}{q^4 - q^3 - q^2 +1}.
\]

\end{proposition}

\begin{proof}
Clearly, $|v({\mathcal{A}}_G)| = q^4 - q^3 - q^2 +1$. We have $|G| = (q^2 -1)(q^2 - q)$ and $|Z(G)| = q - 1$. By Proposition 3.26 of  \cite{Ab06}, the set of centralizers of non-central elements of $GL(2, q)$ is given by
\[
\{xDx^{-1}, xIx^{-1}, xPZ(GL(2, q))x^{-1} : x \in GL(2, q)\}
\]
where $D$ is the subgroup of $GL(2, q)$ consisting  of all diagonal matrices, $I$ is a  cyclic subgroup of $GL(2, q)$ having order $q^2 - 1$  and $P$ is the Sylow $p$-subgroup of $GL(2, q)$ consisting of all upper triangular matrices with $1$ in the diagonal. The orders of  $D$ and $PZ(GL(2, q))$ are  $(q - 1)^2$ and $q(q - 1)$ respectively. Also   the number of conjugates of $D, I$ and $PZ(GL(2, q))$ in $GL(2, q)$  are  $\frac{q(q + 1)}{2}, \frac{q(q - 1)}{2}$ and $q + 1$ respectively. Hence the commuting graph of $GL(2, q)$ is given by
\[
 \frac{q(q + 1)}{2}K_{|xDx^{-1}| - q + 1} \sqcup \frac{q(q - 1)}{2}K_{|xIx^{-1}| - q + 1} \sqcup (q + 1)K_{|xPZ(GL(2, q))x^{-1}| - q + 1}.
\]
That is, ${\overline{{\mathcal{A}}}_G} = \frac{q(q + 1)}{2}K_{q^2 - 3q + 2} \sqcup \frac{q(q - 1)}{2}K_{q^2 - q} \sqcup (q + 1)K_{q^2 - 2q + 1}$. Hence $e({\mathcal{A}}_G) = \frac{q^8 - 2q^7 - 2q^6 + 5q^5 + q^4 - 4q^3 + q}{2}$. By Proposition 4.4 of \cite{lspac}, we have
\begin{align*}
\L-spec({\mathcal{A}}_G) = & \{0, {(q^4-q^3-2q^2+2q)}^{q^3-q^2-2q}, {(q^4-q^3-2q^2+q+1)}^{\frac {q^4-2q^3+q}{2}},\\ & {(q^4-q^3-2q^2+3q-1)}^{\frac {q^4-2q^3-2q^2+q}{2}}, {(q^4-q^3-q^2+1)}^{q^2+q}\}.
\end{align*}
Now, $\left|0 - \frac{2|e({\mathcal{A}}_{G})|}{|v({\mathcal{A}}_G)|}\right| = \frac{q^8 - 2q^7 - 2q^6 + 5q^5 + q^4 - 4q^3 + q}{q^4 - q^3 - q^2 +1}$, $\left|q^4-q^3-2q^2+2q - \frac{2|e({\mathcal{A}}_{G})|}{|v({\mathcal{A}}_{G})|}\right|$\\ $= \frac{q^3 - 2q^2 + q}{q^4 - q^3 - q^2 +1}$, $\left|q^4-q^3-2q^2+q+1 - \frac{2|e({\mathcal{A}}_G)|}{|v({\mathcal{A}}_{G})|}\right| = \frac{q^5 - 2q^4 - q^3 + 3q^2 - 1}{q^4 - q^3 - q^2 +1}$,\\  $\left|q^4-q^3-2q^2+3q-1 - \frac{2|e({\mathcal{A}}_{G})|}{|v({\mathcal{A}}_{G})|}\right| = \frac{q^5 - 2q^4 + q^3 - q^2 + 2q - 1}{q^4 - q^3 - q^2 +1}$ and  \\$\left|q^4-q^3-q^2+1 - \frac{2|e({\mathcal{A}}_{G})|}{|v({\mathcal{A}}_{G})|}\right|= \frac{q^6 - 3q^5 + 2q^4 + 2q^3 - q^2 - q + 1}{q^4 - q^3 - q^2 +1}$. By \eqref{Lenergy}, we have
\begin{align*}
  LE({\mathcal{A}}_G) = & \frac{q^8 - 2q^7 - 2q^6 + 5q^5 + q^4 - 4q^3 + q}{q^4 - q^3 - q^2 +1} + (q^3-q^2-2q)\left(\frac{q^3 - 2q^2 + q}{q^4 - q^3 - q^2 +1}\right)\\
  & + \left(\frac {q^4-2q^3+q}{2}\right)\left(\frac{q^5 - 2q^4 - q^3 + 3q^2 - 1}{q^4 - q^3 - q^2 +1}\right) \\
   & + \left(\frac{q^4-2q^3-2q^2+q}{2}\right)\left(\frac{q^5 - 2q^4 + q^3 - q^2 + 2q - 1}{q^4 - q^3 - q^2 +1}\right)
\end{align*}
and hence the result follows.
\end{proof}
\begin{proposition}\label{Hanaki1}
Let $F = GF(2^n), n \geq 2$ and $\vartheta$ be the Frobenius  automorphism of $F$, i. e., $\vartheta(x) = x^2$ for all $x \in F$. If $G$ denotes the group
\[
 \left\lbrace U(a, b) = \begin{bmatrix}
        1 & 0 & 0\\
        a & 1 & 0\\
        b & \vartheta(a) & 1
       \end{bmatrix} : a, b \in F \right\rbrace
\]
under matrix multiplication given by $U(a, b)U(a', b') = U(a + a', b + b' + a'\vartheta(a))$, then
\[
LE({\mathcal{A}}_G) = 2^{2n + 1} - 2^{n + 2}.
\]
\end{proposition}

\begin{proof}
It is clear that $v({\mathcal{A}}_G) = 2^n(2^n - 1)$. Note that $Z(G) = \{U(0, b) : b\in F\}$ and so $|Z(G)| = 2^n$. Let $U(a, b)$ be a non-central element of $G$. The centralizer of $U(a, b)$ in $G$ is $Z(G)\sqcup U(a, 0)Z(G)$. Hence $\overline{{\mathcal{A}}}_{G} = (2^n - 1)K_{2^n}$ and $|e({\mathcal{A}}_G)| = \frac{2^{4n} - 3.2^{3n} + 2.2^{2n}}{2}$. By Proposition 4.5 of \cite{lspac}, we have
\[
\L-spec({\mathcal{A}}_G) = \{ 0, {(2^{2n}-2^{n+1})}^{{(2^n-1)}^2}, {(2^{2n}-2^n)}^{2^n-2} \}.
\]
Therefore, $\left|0 - \frac{2|e({\mathcal{A}}_{G})|}{|v({\mathcal{A}}_G)|}\right| = 2^{2n} - 2.2^n$, $\left|2^{2n}-2^{n+1} - \frac{2|e({\mathcal{A}}_G)|}{|v({\mathcal{A}}_{G})|}\right| = 0$  and \\ $\left|2^{2n}-2^n - \frac{2|e({\mathcal{A}}_{G})|}{|v({\mathcal{A}}_{G})|}\right| = 2^n$. By \eqref{Lenergy}, we have
\[
LE({\mathcal{A}}_G) = 2^{2n} - 2.2^n + ({(2^n-1)}^2)0 + (2^n-2)2^n
\]
and hence the result follows.
\end{proof}
\begin{proposition}\label{Hanaki2}
Let $F = GF(p^n)$ where $p$ is a prime. If $G$ denotes    the  group
\[
\left\lbrace V(a, b, c) = \begin{bmatrix}
        1 & 0 & 0\\
        a & 1 & 0\\
        b & c & 1
       \end{bmatrix} : a, b, c \in F \right\rbrace
\]
under matrix multiplication $V(a, b, c)V(a', b', c') = V(a + a', b + b' + ca', c + c')$,
then
\[
LE({\mathcal{A}}_G) = 2(p^{3n} - p^{2n}).
\]
\end{proposition}
\begin{proof}
Clearly, $v({\mathcal{A}}_G) = p^n(p^{2n} - 1)$. We have $Z(G) = \{V(0, b, 0) : b \in F\}$ and so $|Z(G)| = p^n$. The centralizers of non-central elements of $A(n, p)$ are given by
\begin{enumerate}
\item If $b, c \in F$ and $c \ne 0$ then the centralizer of $V(0, b, c)$ in $G$ is\\ $\{V(0, b', c') : b', c' \in F\}$ having order $p^{2n}$.
\item If $a, b \in F$ and $a \ne 0$ then the centralizer of $V(a, b, 0)$ in $G$ is\\ $\{V(a', b', 0) : a', b' \in F\}$ having order $p^{2n}$.
\item If $a, b, c \in F$ and $a \ne 0, c \ne 0$ then the centralizer of $V(a, b, c)$ in $G$ is \\$\{V(a', b', ca'a^{-1}) : a', b' \in F\}$ having order $p^{2n}$.
\end{enumerate}
It can be seen that all the centralizers of non-central elements of $A(n, p)$ are abelian. Hence,
\[
\overline{{\mathcal{A}}}_G = K_{p^{2n} - p^n}\sqcup K_{p^{2n} - p^n}\sqcup (p^n - 1)K_{p^{2n} - p^n} = (p^n + 1)K_{p^{2n} - p^n}.
\]
and $e({\mathcal{A}}_G) = \frac{p^{6n} - p^{5n} - p^{4n} + p^{3n}}{2}$. By Proposition 4.6 of \cite{lspac}, we have
\[
\L-spec({\mathcal{A}}_{A(n, p)}) = \{ 0, {(p^{3n}-p^{2n})}^{p^{3n}-2p^{n}-1}, {(p^{3n}-p^{n})}^{p^n} \}.
\]
Therefore, $\left|0 - \frac{2|e({\mathcal{A}}_{G})|}{|v({\mathcal{A}}_G)|}\right| = p^{3n} - p^{2n}$, $\left|p^{3n}-p^{2n} - \frac{2|e({\mathcal{A}}_G)|}{|v({\mathcal{A}}_{G})|}\right| = 0$  and \\ $\left|p^{3n}-p^{n} - \frac{2|e({\mathcal{A}}_{G})|}{|v({\mathcal{A}}_{G})|}\right| = p^{2n} - p^n$. By \eqref{Lenergy}, we have
\[
LE({\mathcal{A}}_G) = p^{3n} - p^{2n} + (p^{3n}-2p^{n}-1)0 + p^n(p^{2n} - p^n)
\]
and hence the result follows.
\end{proof}

\section{Some consequences}

For a finite group $G$, the set $C_G(x) = \{y \in G : xy = yx\}$ is called the centralizer of an element $x \in G$. Let $|\cent(G)| = |\{C_G(x) : x \in G\}|$, that is the number of distinct centralizers in $G$. A group $G$ is called an $n$-centralizer group if $|\cent(G)| = n$. The study of these groups was initiated by  Belcastro and  Sherman   \cite{bG94} in the year 1994. The readers may conf. \cite{Dutta10} for various results on these groups. In this section, we compute Laplacian energy of the non-commuting graphs of non-abelian  $n$-centralizer finite groups for some positive integer $n$.   We begin with the following result.

\begin{proposition}\label{4-cent}
If $G$ is a finite $4$-centralizer group, then
\[
LE({\mathcal{A}}_G) = 4|Z(G)|.
\]
\end{proposition}
\begin{proof}
Let $G$ be a finite $4$-centralizer group. Then, by  \cite[Theorem 2]{bG94}, we have  $\frac{G}{Z(G)} \cong {\mathbb{Z}}_2 \times {\mathbb{Z}}_2$. Therefore, by Theorem \ref{main2}, the result follows.
\end{proof}

\noindent Further, we have the following result.

\begin{corollary}
If $G$ is a  finite $(p+2)$-centralizer $p$-group for any prime $p$,  then
\[
LE({\mathcal{A}}_G) = 2p(p - 1)|Z(G)|.
\]
\end{corollary}
\begin{proof}
Let $G$ be a finite $(p + 2)$-centralizer $p$-group. Then, by   \cite[Lemma 2.7]{ali00}, we have  $\frac{G}{Z(G)} \cong {\mathbb{Z}}_p \times {\mathbb{Z}}_p$. Therefore, by Theorem \ref{main2}, the result follows.
\end{proof}

\begin{proposition}\label{5-cent}
If $G$ is a  finite $5$-centralizer  group, then

\[
LE({\mathcal{A}}_G) = 12|Z(G)| \text{ or } \frac{18{|Z(G)|}^2 + 27|Z(G)|}{5}.
\]

\end{proposition}

\begin{proof}
Let $G$ be a finite $5$-centralizer group. Then by  \cite[Theorem 4]{bG94}, we have  $\frac{G}{Z(G)} \cong {\mathbb{Z}}_3 \times {\mathbb{Z}}_3$ or $D_6$. Now, if $\frac{G}{Z(G)} \cong {\mathbb{Z}}_3 \times {\mathbb{Z}}_3$, then  by Theorem \ref{main2}, we have
\[
LE({\mathcal{A}}_G) = 12|Z(G)|.
\]
If $\frac{G}{Z(G)} \cong D_6$, then by Theorem \ref{main4} we have
\[
LE({\mathcal{A}}_G) = \frac{18{|Z(G)|}^2 + 27|Z(G)|}{5}.
\]

\end{proof}

%\noindent We conclude this section with the  following corollary.
%\begin{corollary}
%Let $G$ be a finite non-abelian group and $\{x_1, x_2, \dots, x_r\}$ be a set of pairwise non-commuting elements of $G$ having maximal size. Then $G$ is super integral if $r = 3, 4$.
%\end{corollary}

%\begin{proof}
%By Lemma 2.4 in \cite{ajH07}, we have that $G$ is a $4$-centralizer or a $5$-centralizer group according as  $r = 3$ or $4$. Hence the result follows from Theorem \ref{4-cent} and Theorem \ref{5-cent}.
%\end{proof}
%\section{Commutativity degree and super integral group}
Let $G$ be a finite group. The commutativity degree of $G$ is given by the ratio
\[
\Pr(G) = \frac{|\{(x, y) \in G \times G : xy = yx\}|}{|G|^2}.
\]
The origin of the commutativity degree of a finite group lies in a paper of Erd$\ddot{\rm o}$s and Tur$\acute{\rm a}$n (see \cite{Et68}). Readers may conf. \cite{Caste10,Dnp13,Nath08} for various results on $\Pr(G)$. In  the following few results we shall compute various energies of the commuting graphs of finite non-abelian groups $G$ such that $\Pr(G) = r$ for some rational number $r$.

\begin{proposition}
Let $G$ be a finite group and $p$ the smallest prime divisor of $|G|$. If $\Pr(G) = \frac{p^2 + p - 1}{p^3}$, then

\[
LE({\mathcal{A}}_G) = 2p(p - 1)|Z(G)|.
\]

\end{proposition}
\begin{proof}
If $\Pr(G) = \frac{p^2 + p - 1}{p^3}$, then by \cite[Theorem 3]{dM74}, we have $\frac{G}{Z(G)}$ is isomorphic to ${\mathbb{Z}}_p\times {\mathbb{Z}}_p$. Hence the result follows from  Theorem \ref{main2}.
\end{proof}
As a corollary we have
\begin{corollary}
Let $G$ be a finite group such that $\Pr(G) = \frac{5}{8}$. Then
\[
LE({\mathcal{A}}_G) = 4|Z(G)|.
\]
\end{corollary}

\begin{proposition}
If $\Pr(G) \in \{\frac{5}{14}, \frac{2}{5}, \frac{11}{27}, \frac{1}{2}\}$, then
\[
LE({\mathcal{A}}_G) = 9,\frac{28}{3},25 \text{ or } \frac{126}{5}.
\]
\end{proposition}
\begin{proof}
If $\Pr(G) \in \{\frac{5}{14}, \frac{2}{5}, \frac{11}{27}, \frac{1}{2}\}$, then as shown in \cite[pp. 246]{Rusin79} and \cite[pp. 451]{Nath13}, we have $\frac{G}{Z(G)}$ is isomorphic to one of the groups in $\{D_6, D_8, D_{10}  D_{14}\}$. Hence the result follows from Corollary \ref{main005}.
\end{proof}
%We conclude this section by the following result.
%\begin{theorem}
%If $G$ is a non-solvable group with $\Pr(G) = \frac{1}{12}$ then $G$ is super integral.
%\end{theorem}
%\begin{proof}
%By Proposition 3.3.7 in \cite{Caste10} we have that $G$ is isomorphic to $A_5 \times B$ for some abelian group $B$. Therefore $G$ is an AC-group and hence super integral.
%\end{proof}

%\section{More Applications}
\begin{proposition}\label{order16}
Let $G$ be a group isomorphic to any of the following groups
\begin{enumerate}
\item ${\mathbb{Z}}_2 \times D_8$
\item ${\mathbb{Z}}_2 \times Q_8$
\item $M_{16}  = \langle a, b : a^8 = b^2 = 1, bab = a^5 \rangle$
\item ${\mathbb{Z}}_4 \rtimes {\mathbb{Z}}_4 = \langle a, b : a^4 = b^4 = 1, bab^{-1} = a^{-1} \rangle$
\item $D_8 * {\mathbb{Z}}_4 = \langle a, b, c : a^4 = b^2 = c^2 =  1, ab = ba, ac = ca, bc = a^2cb \rangle$
\item $SG(16, 3)  = \langle a, b : a^4 = b^4 = 1, ab = b^{-1}a^{-1}, ab^{-1} = ba^{-1}\rangle$.
\end{enumerate}
Then
\[
LE({\mathcal{A}}_G) = 16.
\]

\end{proposition}
\begin{proof}
If $G$ is isomorphic to any of the above listed  groups, then $|G| = 16$ and $|Z(G)| = 4$. Therefore, $\frac{G}{Z(G)} \cong {\mathbb{Z}}_2 \times {\mathbb{Z}}_2$. Thus the result follows from Theorem~\ref{main2}.
\end{proof}

Recall that  genus of a graph is the smallest non-negative integer $n$ such that the graph can be embedded on the surface obtained by attaching $n$ handles to a sphere. A graph is said to be planar  if the genus of the graph is zero. We conclude this paper with the following result.

\begin{theorem}
Let $\Gamma_G$ be the commuting graph of a finite non-abelian  group $G$.  If  $\Gamma_G$  is planar   then
\[
LE({\mathcal{A}}_G)  = \frac{28}{3} \text{ or } 9.
\]
\end{theorem}

\begin{proof}
From \cite[Proposition 2.3]{AF14}, if ${\mathcal{A}}_G$ is planner then $G \cong S_3, D_8 \text{ or } Q_8$. From Corollary \ref{main005} and Corollary \ref{q4m}, if $G\cong D_8 \text{ or } Q_8$ then $LE({\mathcal{A}}_G) = \frac{28}{3}$ and if $G \cong S_3$ then $LE({\mathcal{A}}_G) = 9$.
\end{proof}

\end{document}